\documentclass[11pt]{amsart}

\usepackage{amssymb}
\usepackage{amsmath}
\usepackage{amsthm}
\usepackage{enumerate}
\usepackage{graphicx}
\usepackage{caption}
\usepackage{subcaption}
\usepackage{hyperref}  
\hypersetup{colorlinks=true, citecolor=blue}

\theoremstyle{plain}
\newtheorem{thm}{Theorem}[section]
\newtheorem{prop}[thm]{Proposition}
\newtheorem{lem}[thm]{Lemma}
\newtheorem{cor}[thm]{Corollary}
\newtheorem{question}[thm]{Question}
\newtheorem{conj}[thm]{Conjecture}

\numberwithin{equation}{section}

\theoremstyle{definition}
\newtheorem{defn}[thm]{Definition}
\newtheorem{ex}[thm]{Example}

\theoremstyle{remark}
\newtheorem{remark}[thm]{Remark}

\newtheorem*{acknowledgements}{Acknowledgements}

\theoremstyle{plain}

\newcommand{\thmref}[1]{Theorem~\ref{#1}}
\newcommand{\conjref}[1]{Conjecture~\ref{#1}}
\newcommand{\propref}[1]{Proposition~\ref{#1}}
\newcommand{\secref}[1]{Section~\ref{#1}}

\newcommand{\lemref}[1]{Lemma~\ref{#1}}
\newcommand{\corref}[1]{Corollary~\ref{#1}}
\newcommand{\figref}[1]{Figure~\ref{#1}}
\newcommand{\exref}[1]{Example~\ref{#1}}

\newcommand{\eqnref}[1]{Equation~\eqref{#1}}

\newcommand{\M}{{\mathcal M}}

\newcommand{\calQ}{{\mathcal Q}}
\newcommand{\calR}{{\mathcal R}}
\newcommand{\calS}{{\mathcal S}}

\newcommand{\T}{{\mathcal T}}

\newcommand{\calX}{{\mathcal X}}

\newcommand{\X}{{\calX}}

\newcommand{\R}{{\mathbb R}}
\newcommand{\Z}{{\mathbb Z}}

\newcommand{\wtilde}{\widetilde}

\newcommand{\eps}{{\varepsilon}}

\DeclareMathOperator{\isom}{Isom}

\DeclareMathOperator{\sys}{sys}

\DeclareMathOperator{\starr}{st}

\begin{document}

\title{Hyperbolic surfaces with sublinearly many systoles that fill}

\author{Maxime Fortier Bourque}
\address{School of Mathematics and Statistics, University of Glasgow, University Place, Glasgow, United Kingdom, G12 8QQ}
\email{maxime.fortier-bourque@glasgow.ac.uk}

\begin{abstract}
For any $\eps>0$, we construct a closed hyperbolic surface of genus $g=g(\eps)$ with a set of at most $\eps g$ systoles that fill, meaning that each component of the complement of their union is contractible. This surface is also a critical point of index at most $\eps g$ for the systole function, disproving the lower bound of $2g-1$ posited by Schmutz Schaller.
\end{abstract}

\maketitle

\section{Introduction}

The moduli space $\M_{g,n}$ of Riemann surfaces of genus $g$ with $n$ punctures  is an object of great interest to many geometers and topologists. It encodes all the different complex structures, conformal structures, or hyperbolic structures (provided $2g+n > 2$) supported on a surface with given topology. The topology of moduli space is largely encoded in its orbifold fundamental group $\Gamma_{g,n}$, the mapping class group. 

All the torsion-free finite index subgroups of $\Gamma_{g,n}$ have the same cohomological dimension, which is called the \emph{virtual cohomological dimension} (vcd) of $\Gamma_{g,n}$. Harer computed the vcd of $\Gamma_{g,n}$ for all $g$ and $n \geq 0$ and found a spine (a deformation retract) for $\M_{g,n}$ with this smal\-lest possible dimension whenever $n>0$ \cite{Harer}. When $n=0$, the vcd of the mapping class group is equal to $4g-5$, but a spine of this dimension has yet to be found \cite[Question 1]{BridsonVogtmann}. The largest codimension attained so far is equal to $2$ \cite{Ji} (the space $\M_{g,0}$ has dimension $6g-6$).

In an unpublished preprint \cite{Thurston}, Thurston claimed that the set $\X_g$ of closed hyperbolic surfaces of genus $g\geq 2$ whose systoles fill forms a spine for $\M_g = \M_{g,0}$. Recall that a \emph{systole} is a closed geodesic of minimal length, and a set of curves \emph{fills} if each component of the complement of their union is simply connected. Thurston's proof that $\M_g$ deformation retracts onto $\X_g$ appears to be difficult to complete \cite{Ji}. Furthermore, the dimension of $\X_g$ is still not known, mostly because we do not understand which filling sets of curves can be systoles. Indeed, Thurston writes: \begin{quote}Unfortunately, we do not have a combinatorial characterization of collections of curves which can be the collection of shortest geodesics on a surface. This seems like a challenging problem, and until more is understood about how to answer it, there are probably not many applications of the current result.\end{quote}

The paper \cite{APP} provides some partial answers to Thurston's question. On a closed hyperbolic surface, systoles do not self-intersect and distinct systoles can intersect at most once. This obvious necessary condition is, however, far from being sufficient. Indeed, a filling set of systoles must contain at least $\sim \pi g / \log g$ curves \cite[Theorem 3]{APP}, but there exist filling collections of $\sim2\sqrt{g}$ geodesics pairwise intersecting at most once \cite[Corollary 2]{APP}. There is a discrepancy in the opposite direction as well: a closed hyperbolic surface can have at most $C g^2 / \log g$ systoles \cite[Corollary 1.4]{ParlierKissing}, but there exist filling collections with more than $g^2$ geodesics pairwise intersecting at most once \cite[Theorem 1.1]{Malestein}.

Our main result here is a construction of closed hyperbolic surfaces with filling sets of systoles containing sublinearly many curves in terms of the genus. Compare this with \cite{SchmutzManySystoles,SchmutzMoreSystoles} where surfaces with superlinearly many systoles are found. Though we are still very far\footnote{The genus $g$ in \thmref{thm:fewsystoles} grows like a tower of exponentials of length roughly $1/\eps$.} from the lower bound of $\pi g / \log g$, our examples improve upon the previous record of surfaces with filling sets of $2g$ systoles \cite[Section 5]{APP} \cite{Sanki}.

\begin{thm} \label{thm:fewsystoles}
For every $\eps>0$, there exist an integer $g\geq 2$ and a closed hyperbolic surface of genus $g$ with a filling set of at most $\eps g$ systoles.
\end{thm}

Near a surface with a filling set of at most $\eps g$ systoles, Thurston's set $\X_g$ contains the set of solutions to the same number of equations. This should imply that $\X_g$ has codimension at most $\eps g$ in $\M_g$. However, the equations requiring the curves to have equal length can be redundant, preventing us from applying the implicit function theorem. We only manage to prove that $\X_g$ has dimension at least $4g-5$ when $g$ is even, but conjecture the following.

\begin{conj} \label{conj:dimension}
For every $\eps>0$, there exists an integer $g\geq 2$ such that $\X_g$ has dimension at least $(6-\eps)g$. 
\end{conj}

On the other hand, we can prove that a closely related spine, the Morse--Smale complex for the systole function, has dimension much larger than the virtual cohomological dimension of the mapping class group.

In a series of papers \cite{SchmutzMaxima,SchmutzMoreExamples,SchmutzSurvey,SchmutzMorse}, Schmutz Schaller initiated the study of the systole function $\sys: \M_{g,n} \to \R_+$, which records the length of any of the shortest closed geodesics on a surface. He proved that the systole function is a topological Morse function on the Teichm\"uller space $\T_{g,n}$ whenever $n>0$ \cite{SchmutzMorse} and Akrout  extended this result to $n=0$ (and to a more general class of functions)  in \cite{Akrout}. 

Schmutz Schaller constructed a critical point of index $2g-1$ for the systole function in every genus $g \geq 2$ and thought it was ``quite possible'' that this was the smallest achievable index \cite[p.439]{SchmutzMorse}. He verified this hypothesis for $g=2$ by finding all the critical points in $\M_2$. If this were true in general, it would imply that the Morse--Smale complex for the systole function has the smallest possible dimension $4g-5 = (6g-6)-(2g-1)$ for a spine of $\M_g$. However, our surfaces show that no such inequality holds.

\begin{thm} \label{thm:critical}
For every $\eps>0$, there exist an integer $g\geq 2$ and a critical point of index at most $\eps g$ for the systole function  on $\T_g$.
\end{thm}

\subsection*{Organization} The surfaces arising in Theorems \ref{thm:fewsystoles} and \ref{thm:critical} are built in two steps, in a similar fashion as the local maxima from \cite{localmax}. First,  in \secref{sec:block}, we define a buil\-ding block (depending on some parameters) which is a surface whose systoles are the boundary components. This surface is modelled on a flag-transitive surface map (a generalization of Platonic solids) and can be cut into isome\-tric right-angled polygons along a collection of geodesic arcs. We then glue building blocks together according to the combinatorics of certain graphs of large girth with strong transitivity properties in \secref{sec:gluing}. We do this in such a way that the boundaries of the blocks remain systoles in the larger surface and that the arcs in the blocks connect up to form systoles as well (see \secref{sec:systoles}). In \secref{sec:isometries}, we show that $X / \isom(X)$ is isometric to a triangle or a quadrilateral. This easily implies that $X$ is a critical point of the systole function, which we prove in \secref{sec:critical}. Finally, we discuss our failed attempt to prove \conjref{conj:dimension} in \secref{sec:dimension}. 

\begin{acknowledgements}
I thank the anonymous referee for their useful comments and corrections.
\end{acknowledgements}

\section{Building blocks}  \label{sec:block}

\subsection*{Graphs}

A \emph{graph} is a $1$-dimensional cell complex, where there can be multiple edges between two vertices and edges from vertices to themselves. The \emph{valence} of a vertex in a graph is the number of half-edges adjacent to it. If every vertex in a graph has the same valence, then this number is called the valence of the graph. 

\subsection*{Flag-transitive maps}

A \emph{map} $M$ is a graph embedded on a surface $S$ such that the closure of each complementary component is an embedded closed disk (called a \emph{face} of the map). All maps considered in this paper will be \emph{orientable}, meaning that the surface $S$ is required to be orientable. If all the faces of a map $M$ have the same number $p$ of edges and all the vertices have the same valence $q$, then $M$ is said to have \emph{type} $\{p,q\}$.

A \emph{flag} in a map is a triple consisting of a vertex $v$, an edge $e$ containing $v$, and a face $f$ containing $e$. A \emph{map-automorphism} of $M$ is an automorphism of the underlying graph which can be realized by a homeomorphism of the surface $S$. A map is \emph{flag-transitive}\footnote{These maps are usually called \emph{regular}, but if we stuck to standard terminology, this word would be used with five different meanings throughout the paper.} if its group of map-automorphisms acts transitively on flags. 

Any flag-transitive map has type $\{p,q\}$ for some $p\geq 1$ and $q\geq 2$. The five Platonic solids are the only flag-transitive maps on the sphere with $p,q\geq 3$; their types are $\{3,3\}$, $\{4,3\}$, $\{3,4\}$, $\{5,3\}$ and $\{3,5\}$. Beach balls assembled from $q$ spherical bigons are flag-transitive maps of type $\{2,q\}$.

\subsection*{Maps of large girth}

A \emph{cycle} in a graph is a sequence of oriented edges $(e_1,\ldots,e_k)$ such that the endpoint of $e_i$ coincides with the starting point of $e_{i+1}$ for every $i \in \Z_k$. Cycles are considered up to cyclic permutation of their edges and reversal of orientation. The \emph{length} of a cycle is the number of edges that it uses. A cycle  is \emph{non-trivial} if it cannot be homotoped to a point by deleting \emph{backtracks}, that is, consecutive edges (modulo $k$) with opposite orientations. The \emph{girth} of a graph is the length of any of its shortest non-trivial cycles. These shortest non-trivial cycles will be called \emph{girth cycles}. A graph of girth at most $2$ is often called a \emph{multigraph}, and a graph of girth larger than $2$ is \emph{simple}.

The girth of a flag-transitive map $M$ of type $\{p,q\}$ is at most $p$ since the faces are non-trivial cycles of length $p$. If $M$ is finite, then one can actually unwrap all the cycles shorter than $p$ by taking a suitable finite normal cover, thereby obtaining a finite flag-transitive map $N$ of girth $p$ \cite[Theorem 11]{Evans}. That such covers exist follows from Mal'cev's theorem on the residual finiteness of finitely generated linear groups \cite{Malcev}.

\begin{thm}[Evans] \label{thm:reg_map}
For any $p,q \geq 2$, there exists a finite flag-transitive map of type $\{p,q\}$ and girth $p$.
\end{thm}

See also \cite{Nedela} and \cite{Siran} for constructive proofs of this result.

\subsection*{Regular polygons}

Let $q\geq 3$. Up to isometry, there exists a unique polygon $P$ in the hyperbolic plane with $2q$ sides of the same length $L$ and all interior angles equal to $\pi/2$. We will call $P$ the \emph{regular right-angled $2q$-gon}. By connecting the center of $P$ to the midpoint of a side and one of its vertices, we obtain a triangle with interior angles $\pi/2$, $\pi/4$ and $\pi/2q$ and a side of length $L/2$ from which we obtain the equation
\begin{equation} \label{eq:trig}
\cosh(L/2)= \cos(\pi/2q)  / \sin(\pi/4) = \sqrt{2}  \cos(\pi/2q)
\end{equation}
(see \cite[p.454]{Buser}). We color the sides of $P$ red and blue in such a way that adjacent sides have different colors.

\begin{lem} \label{lem:arc_in_poly}
Any arc $\alpha$ between two disjoint sides in the regular right-angled $2q$-gon $P$ has length at least $L$, with equality only if $\alpha$ is a side of $P$.
\end{lem}
\begin{proof}
Let $\alpha$ have minimal length among such arcs. Then $\alpha$ must be geodesic and orthogonal to $\partial P$ at its endpoints.  These endpoints are separated by $m$ sides of $P$ in one direction and $n$ sides in the other, where $m+n+2=2q$ and $m\leq n$. First suppose that $m>1$.  Let $d$ be a main diagonal of $P$ which is linked with $\alpha$ and has one endpoint at an extremity of one of the two sides of $P$ joined by $\alpha$. Let $z$ be the intersection point between $\alpha$ and $d$, and let $\alpha_\pm$ be the two components of $\alpha \setminus\{ z\}$ labelled in such a way that $\alpha_+$ and $d$ have endpoints in a common side of $P$. If $R_d$ denotes the reflection about $d$, then the arc $\gamma=\alpha_- \cup R_d(\alpha_+)$ has the same length as $\alpha$ and joins two disjoint sides of $P$ (because $m>1$). By minimality, $\gamma$ must be geodesic and orthogonal to $\partial P$, which is absurd. This shows that $m=1$, in which case $\alpha$ is a side of $P$ (the orthogonal segment between two geodesics in the hyperbolic plane is unique when it exists).  
\end{proof}

One can also prove this using trigonometry (see \cite[p.91]{APP}).

\subsection*{Gluing regular polygons along maps}

Let $M$ be an oriented map of type $\{p,q\}$ where $q\geq 3$. Let $P$ be the unique right-angled regular hyperbolic $2q$-gon with sides colored red and blue as above. We now define a hyperbolic surface $B$ modelled on $M$. For each vertex $v\in M$, take a copy $P_v$ of $P$. The blue sides of $P_v$ are labelled in counterclockwise order by the edges adjacent to $v$ in $M$, which come with a cyclic ordering from the orientation. For each edge $e=\{u,v\}$ in $M$, we glue the polygons $P_u$ and $P_v$ along their sides labelled $e$ by an orientation-reversing isometry. The resulting surface is denoted $B$ and will be called a \emph{block} in the sequel. The polygons $P_v \subset B$ are its \emph{tiles}.

Topologically, $B$ is the same as the surface $S \supset M$ with a hole cut out in each face. Indeed, if we join the center of each polygon $P_v$ to the midpoints of its blue sides, we obtain an embedded copy of $M$ in $B$. Since each $P_v$ deformation retracts onto the star $M\cap P_v$, the surface $B$ deformation retracts onto $M$. Each boundary component of $B$ is the concatenation of $p$ red sides of polygons $P_v$ coming from the $p$ vertices $v$ around a face of $M$. In particular, each boundary component of $B$ has length $p L$, where $L$ is the positive number implicitly defined by \eqnref{eq:trig}.

\begin{lem} \label{lem:systoles_in_block}
Let $M$ be a map of type $\{p,q\}$ and girth $p$, where $p \geq 2$ and $q \geq 3$. Then the systoles in $B$ are the boundary geodesics, of length $p L$.
\end{lem}
\begin{proof}
Let $\gamma$ be a systole in $B$. As explained above, the map $M$ embeds in $B$ as the dual graph to the decomposition into the $2q$-gons $P_v$. Let $\pi :B\to M$ be the nearest point projection. The image $\pi(\gamma)$ must be non-trivial in $M$ since $\gamma$ is non-trivial in $B$ and $\pi$ is a deformation retraction. It follows that the combinatorial length of $\pi(\gamma)$ in $M$ is at least $p$. In other words, $\gamma$ intersects at least $p$ tiles $P_v$, joining distinct blue sides each time.  By \lemref{lem:arc_in_poly}, the length of $\gamma \cap P_v$ is at least $L$ for any tile $P_v$ that $\gamma$ intersects. The total length of $\gamma$ is therefore greater than or equal to $p L$. If equality occurs, then $\gamma$ must be a concatenation of red arcs, that is, a boundary geodesic.
\end{proof}

\begin{cor} \label{cor:arc_boundary_to_self}
Let $M$ be a map of type $\{p,q\}$ and girth $p$, where $p \geq 2$ and $q \geq 3$. Then any arc from a boundary component to itself in $B$ which cannot be homotoped into $\partial B$ has length strictly larger than $p L/2$.
\end{cor}
\begin{proof}
Suppose that $\alpha$ is a non-trivial arc of length at most $p L/2$ from a boundary geodesic $b$ to itself. The arc $\alpha$  followed by the shorter of the two subarcs of $b$ between its endpoints is a non-trivial closed curve $\gamma$ of length at most $p L$ in $B$. The closed geodesic homotopic to $\gamma$ is strictly shorter, contradicting \lemref{lem:systoles_in_block}.
\end{proof}

\begin{lem} \label{lem:arc_boundary_to_boundary}
Let $M$ be a map of type $\{p,q\}$ and girth $p$, where $p \geq 2$ and $q \geq 3$. Then any arc $\alpha$ from $\partial B$ to $\partial B$ which cannot be homotoped into $\partial B$ has length at least $L$, with equality only if $\alpha$ is a blue arc.
\end{lem}
\begin{proof}
Let $\alpha$ be a geodesic arc from $\partial B$ to $\partial B$. By \lemref{lem:systoles_in_block}, we may assume that $\alpha$ joins consecutive sides of any tile $P_v$ it intersects. Since the starting point of $\alpha$ in on a red side, it has to next intersect a blue side, and then a red. This means that $\alpha$ is homotopic to a blue arc in a union $P_u \cup P_v$ of two adjacent tiles. This blue arc is shortest among all arcs in $P_u \cup P_v$ joining the same two sides, as it is orthogonal to the boundary at both endpoints. 
\end{proof}

The above results do not require the map $M$ to be finite or flag-transitive, but we will impose these conditions in the next sections.

\section{Gluing graphs} \label{sec:gluing}

In this section, we explain how to glue blocks together along certain graphs of large girth with large automorphism groups in order to get closed hyperbolic surfaces with many symmetries and few systoles. 

\subsection*{Strict polygonal graphs}

A \emph{strict polygonal graph} is a gra\-ph $G$ such that any embedded path of length $2$ in $G$ is contained in a unique girth cycle (where cycles are considered up to cyclic reordering and reversal). This notion was introduced by Perkel in his thesis \cite{PerkelThesis}. Examples of strict polygonal graphs include polygons, the tetrahedron, the dodecahedron, and the cube of any dimension. See \cite{Seress} for a short survey on the subject.

Archdeacon and Perkel \cite{ArchPerkel} found a way to double the girth of a strict polygonal graph $G$ (or any graph) by taking an appropriate normal covering space. The girth cycles in this cover $\wtilde G$ are precisely those that wrap twice around a girth cycle in $G$ under the covering map. Repeated applications of their construction yield strict polygonal graphs of arbitrarily large girth and constant valence (equal to the valence of $G$).

Seress and Swartz \cite[Theorem 3.2]{SeressSwartz} proved that any automorphism of the base graph $G$ lifts to an automorphism of the girth-doubling cover $\wtilde G$. They concluded that if $G$ is vertex transitive, edge transitive, arc transitive or 2-arc transitive, then so is $\wtilde G$. We will need an even stronger transitivity property, described in the next paragraph. 

\subsection*{Isotropic graphs}

The \emph{star} $\starr(v)$ of a vertex $v$ in a graph is the set of half-edges adjacent to $v$. A graph $G$ is \emph{locally symmetric} if for every vertex $v \in V(G)$, any bijection of $\starr(v)$ can be extended to an automorphism of $G$ that fixes $v$. We say that a graph is \emph{isotropic} if it is vertex transitive and locally symmetric. To spell it out, $G$ is isotropic if every injection $\starr(u) \hookrightarrow \starr(v)$ between stars in $G$ extends to an automorphism of $G$. 

In an isotropic graph, there is a girth cycle passing through any embedded path of length $2$, but there can be more than one.

\begin{ex}
The Petersen graph $P$  (the quotient of the dodecahedron by the antipodal involution) is an isotropic graph of valence $3$ and girth $5$ on $10$ vertices. However, $P$ is not strict polygonal since every embedded path of length $2$ is contained in two distinct girth cycles in $P$. 
\end{ex}

Lubotzky \cite{Lubotzky} constructed infinitely many isotropic Cayley graphs of any valence $d\geq 3$ and any even girth $\geq 6$ (the generators are involutions, allowing the valence to be odd). Since we want better control on the girth cycles of our isotropic graphs, we use the girth-doubling construction of Archdeacon and Perkel instead. The proof that the girth-doubling cover $\wtilde{G}$ of a graph $G$ is isotropic provided that $G$ is isotropic follows immediately from \cite[Theorem 3.2]{SeressSwartz}, which states that any automorphism of $G$ lifts to $\wtilde{G}$,  and the fact that the covering $\wtilde G \to G$ is normal, so that its deck group acts transitively on fibers.

The simplest isotropic strict polygonal graph is a pair of vertices joined by $d\geq 2$ edges. Repeated applications of the girth-doubling construction to this graph $\Theta$ yield a sequence of finite, isotropic, strict polygonal graphs of any valence and arbitrarily large girth.

\begin{thm}[Archdeacon--Perkel, Seress--Swartz] \label{thm:poly_graph}
For any $d\geq 2$ and $n\geq 1$, there exists a finite, isotropic, strict polygonal graph $G$ of valence $d$ and girth $2^n$. In fact, $G$ can be chosen to be a covering space of the bipartite graph $\Theta$ of valence $d$ on $2$ vertices, in which case the girth cycles in $G$ project to powers of girth cycles in $\Theta$ under the covering map. 
\end{thm}

\subsection*{Gluing}

We now explain how to glue copies of the block $B$ from \secref{sec:block} along a finite isotropic strict polygonal graph $G$ to get a closed hyperbolic surface $X$ with a small set of systoles that fill. 

Let $q\geq 3$, let $n\geq 1$, and write $p=2^n$. Let $M$ be a finite flag-transitive map of type $\{p,q\}$ and girth $p$ whose existence is guaranteed by \thmref{thm:reg_map}. 

Let $B$ be the block obtained by gluing regular right-angled $2q$-gons along the map $M$ as in \secref{sec:block}. Let $d$ be the number of boundary components of $B$, which is is equal to the number of faces in $M$. 

Let $G$ be a finite, isotropic, strict polygonal graph of valence $d$ and girth $p=2^n$ covering the bipartite graph $\Theta$ on two vertices as in \thmref{thm:poly_graph}, and let $\pi : G \to \Theta$ be a covering map. Let $\sigma : V(\Theta) \to \{-1,1\}$ and $\chi: E(\Theta) \to \{1,\ldots,d\}$ be bijections, where $V(\Theta)$ and $E(\Theta)$ are the sets of vertices and edges of $\Theta$ respectively. These induce proper colorings $\sigma\circ \pi$ and $\chi\circ \pi$ of the vertices and edges of $G$ respectively.

For each $v \in V(G)$, let $B_v$ be a copy of the block $B$, equipped with its standard orientation if $\sigma(v)=1$ and with the reverse orientation if $\sigma(v)=-1$. Let $b_1, \ldots , b_d$ be the boundary components of $B$ and label the boundary components of any copy $B_v$ in the same way so that the isometric identification $B_v \cong B$ preserves the indices of boundary components.

Here is how we define the closed hyperbolic surface $X$ given the above combinatorial data. For any edge $e=\{u,v\}$ in $G$, glue $B_u$ to $B_v$ by the identity map along their $j$-th boundary component, where $j=\chi\circ \pi(e)$. The surface $X$ is defined as the quotient of $\sqcup_{v \in V(G)} B_v$ by these gluings. Since the gluing maps are orientation-reversing, $X$ is an oriented surface. It has empty boundary since the coloring $\chi\circ \pi$ takes all values in $\{1,\ldots,d\}$ on the edges containing a given vertex $v$, so that all the boundary components of $B_v$ are glued. Lastly, $X$ is compact because $M$ and $G$ are finite.

The main reason for using strict polygonal graphs in this construction is so that the blue arcs in the blocks $B_v$ all close up to curves of the same length in $X$.

\begin{lem} \label{lem:blue}
Any blue arc in a block $B_v \subset X$ is part of a closed geodesic of length $p L$ in $X$.
\end{lem}
\begin{proof}
Any blue arc $\alpha_v$ in $B_v$ connects two boundary geodesics $b_i$ and $b_j$. The block $B_v$ is glued to two other blocks $B_u$ and $B_w$ via these boundary components, and there are blue arcs $\alpha_u \subset B_u$ and $\alpha_w \subset B_w$ corresponding to $\alpha_v$ under the isometric identifications $B_u \cong B_v \cong B_w$. The concatenation $\alpha_u \cup \alpha_v \cup \alpha_w$ is geodesic since $\alpha_v$ is orthogonal to $\partial B_v$.

By our convention, the arc $\alpha_u$ (resp. $\alpha_w$)  connects the boundary components of $B_u$ (resp. $B_w$) labelled $b_i$ and $b_j$. By repeating the above reflection process with $\alpha_u$ or $\alpha_w$ instead of $\alpha_v$ (and so on), we obtain a bi-infinite path $\delta=(\ldots,u,v,w,\ldots)$ in the graph $G$ whose edges alternate between the colors $i$ and $j$. There is also a bi-infinite geodesic \[\beta = \cdots \cup  \alpha_u \cup \alpha_v \cup \alpha_w \cup \cdots\] in $X$ obtained by concatenating the corresponding blue arcs. 

Since $G$ is a strict polygonal graph, the path $(u,v,w)$ is contained in a unique non-trivial cycle $\gamma$ of length $p$ (the girth of $G$). Furthermore, \thmref{thm:poly_graph} stipulates that $\gamma$ covers a closed cycle of length $2$  in $\Theta$ under the covering map $\pi:G\to\Theta$. This cycle of length $2$ is necessarily formed by the edges $\chi^{-1}(i)$ and $\chi^{-1}(j)$ since $\pi$ respects the coloring of edges. This means that the edges of $\gamma$ alternate between the colors $i$ and $j$, and hence that the path $\delta$ wraps around $\gamma$ periodically in both directions. In other words, $\delta$ closes up after $p$ steps. Similarly, the geodesic $\beta$ is closed and its length is equal to $pL$ since each of its $p$ subarcs has length $L$.
\end{proof}

Note that we have not used the hypotheses that $M$ is flag-transitive nor that $G$ is isotropic yet. This will come up in \secref{sec:isometries} where we determine the isometry group of $X$.

\section{Systoles} \label{sec:systoles}

In this section, we determine and count the systoles in the surface $X$ constructed above.

\begin{prop} \label{prop:count}
Let $X$ be the surface constructed in \secref{sec:gluing}. The systoles in $X$ are the red curves and the blue curves. These systoles fill $X$ and there are $\frac{4q}{(q-2) p}(g-1)$ of them, where $q$ is the valence of the map $M$, $p$ is the girth of $M$ and the gluing graph $G$, and $g$ is the genus of $X$.
\end{prop}
\begin{proof}
Let $\gamma$ be a systole in $X$. If $\gamma$ is contained in a single block $B_v \subset X$, then $\gamma$ is a red curve (of length $pL$) by \lemref{lem:systoles_in_block}. Now assume that $\gamma$ is not contained in any block. Then the blocks $B_{v_1}, \ldots, B_{v_k}$ ($k\geq 2$) that it visits define a closed cycle $s=(v_1,\ldots v_k,v_1)$ in the graph $G$. First suppose that $s$ is trivial in $G$. Then $s$ contains at least two backtracks, that is, vertices $v_j$ in the sequence such that $v_{j-1}=v_{j+1}$. If $s$ backtracks at a vertex $u\in G$, this means that a subarc $\omega$ of $\gamma$ enters and leaves the block $B_u$ via the same boundary component. By \corref{cor:arc_boundary_to_self}, $\omega$ has length strictly larger than $pL/2$. Since there are at least two disjoint subarcs like this, $\gamma$ is longer than $pL$.  We conclude that $s$ is non-trivial in $G$, so that its length is at least $p$, the girth of $G$. But for each vertex $u$ along $s$, the corresponding subarc of $\gamma$ in $B_v$ has length at least $L$ by \lemref{lem:arc_boundary_to_boundary}. Thus the total length of $\gamma$ is at least $pL$. If equality occurs, then $\gamma$ is a concatenation of blue arcs. Conversely, any concatenation of blue arcs has length $pL$  by \lemref{lem:blue}.

The complementary components of the set of systoles in $X$ are precisely the interiors of the tiles from which the blocks are assembled. In particular, the systoles fill.

The number of systoles in $X$ is equal to the total number of red arcs and blue arcs divided by $p$. This is because the red arcs are joined in groups of $p$ to form systoles, and similarly for the blue arcs. Each such arc $\alpha$ (either red or blue) belongs to exactly two tiles. The rhombus with one vertex in the center of each of these two tiles and diagonal $\alpha$ has area $\pi(q-2)/q$ by the Gauss--Bonnet formula (it has two right angles and two angles $\pi/q$). These rhombi tile $X$, which has area $4\pi(g-1)$. Therefore, the number of systoles is $4\pi(g-1)$ divided by $\pi(q-2)/q$, divided by $p$.
\end{proof}

Recall that in the construction of $X$ we could take any $q\geq 3$ and $p=2^n$ for any $n\geq 1$. Given any $\eps>0$, taking $n$ sufficiently large and any $q\geq 3$ gives a surface with a filling set of at most $\eps g$ systoles. This proves \thmref{thm:fewsystoles}. At the other extreme, the largest number of systoles is obtained when $q=3$ and $p=2$, which gives $6g-6$ systoles. By \cite[Theorem 2.8]{SchmutzMaxima}, such a surface has too few systoles to be a local maximum of the systole function, but we will see later that it is nevertheless a critical point of lower index.

\begin{ex} \label{example}
For any $g\geq 2$, if we take the map $M$ to be the bipartite graph of valence $g+1$ on two vertices (as a map on the sphere), then the resulting block $B$ has $g+1$ boundary components. Taking the gluing graph $G$ to be equal to $M$, we obtain a surface $X$ which is the double of $B$ across its boundary. The genus of $X$ is then equal to $g$. Since $q=g+1$ and $p=2$, the number of systoles is $2g+2$ according to the formula in \propref{prop:count}. Removing any two intersecting systoles leaves a filling set of $2g$ systoles. This example was previously described in \cite[Theorem 36]{SchmutzMorse} and \cite[Section 5]{APP} and was the starting point of this paper.
\end{ex}

\begin{remark}
We could allow the graphs $G$ and $M$ to have different girths $p$ and $r$ by replacing the polygons $P$ in the blocks to be semi-regular with side lengths $L_\text{blue}$ and $L_\text{red}$ satisfying $p L_\text{blue} = r  L_\text{red}$. A version of \propref{prop:count} still holds for this generalization, with the count of systoles coming to \[\frac{2q}{(q-2)}\left(\frac{1}{p}+\frac{1}{r}\right)(g-1).\] 
All one has to do is change \lemref{lem:arc_in_poly} to say that the distance between any two blue sides is at least $L_\text{red}$ and the distance between any two red sides is at least $L_\text{blue}$, and modify the other lemmata accordingly.
\end{remark}

\section{Isometries} \label{sec:isometries}

In this section, we determine the isometry group of the surface $X$ up to index $2$. Recall that the blocks $B_v \subset X$ (where $v\in V(G)$) are tiled by regular right-angled $2q$-gons $P_u$ (where $u \in V(M)$). By connecting the center of each polygon $P_u$ to the midpoints of its edges with geodesics, we obtain a tiling $\calQ$ of $X$ by $(2,2,2,q)$-quadrilaterals (i.e., quadrilaterals with three right angles and one angle equal to $\pi/q$). Since any isometry of $X$ preserves the set of systoles, it permutes the complementary polygons $P_u$ and therefore the quadrilaterals in $\calQ$. In fact, any quadrilateral can be sent to any other by an isometry. 

\begin{prop} \label{prop:isom}
The isometry group of $X$ acts transitively on the quadrilaterals in the tiling $\calQ$.
\end{prop}
\begin{proof}
The hypothesis that $M$ is flag-transitive implies that the isometry group of $B$ acts transitively on its $(2,2,2,q)$-quadrilaterals. This is because there is a one-to-one correspondence between the flags in $M$ and the quadrilaterals in $B$. The correspondence works as follows. Recall that $M$ naturally embeds in $B$, connecting the centers of polygons $P$ to their blue sides. A flag in $M$ is the same as a half-edge $e$ together with a choice of a face $f$ containing $e$, either on the left or the right. In the tiling of $B$ by quadrilaterals, there are exactly two quadrilaterals that have $e$ as an edge. The side of $e$ on which $f$ lies determines which quadrilateral to pick. Since any map-automorphism of $M$ can be realized as an isometry of $B$ and $M$ is flag-transitive, the claim follows. 

Let $v \in V(G)$ and let $\phi: B_v \to B_v$ be an isometry. We claim that $\phi$ extends to an isometry $\Phi$ of $X$. First, the isometry $\phi$ induces a permutation $\tau$ on $\{1,\ldots,d\}$ such that $\phi$ sends the boundary component $b_i$ of $B_v$ to the component $b_{\tau(i)}$ for every $i$. Now the edges adjacent to $v$ in $G$ are colored with the numbers $\{1,\ldots,d\}$ according to the coloring $\chi\circ \pi$. Thus the permutation $\tau$ induces a bijection on the star of $v$. Since $G$ is locally symmetric, this bijection can be extended to an automorphism $\psi$ of $G$. If $x \in B_u \subset X$, then define $\Phi(x)$ to be the point $\phi(x)$ in $B_{\psi(u)}$, where we use the canonical identifications $B_w \cong B_v$ to transport the action of $\phi$ onto any block. This map is well-defined, for if $x\in B_u \cap B_v$ then $x$ belongs to the boundary component labelled $i=\chi\circ \pi(\{u,v\})$ of $B_u$ and $B_v$. By definition, $\phi(x)$ belongs to the $\tau(i)$-th boundary component of $B_v$. Now $B_{\psi(u)}$ and $B_{\psi(v)}$ are glued along their boundary component labelled $\chi\circ\pi(\psi(\{u,v\}))$. This number equals $\tau(i)$ provided that the automorphism $\psi$ is chosen to be a lift of the automorphism of $\Theta$ induced by $\tau$, and this is possible according to \cite[Theorem 3.2]{SeressSwartz}. The map $\Phi$ is an isometry since it is a locally isometry as well as a bijection.

Similarly, any automorphism $\psi$ of $G$ which preserves the coloring $\chi\circ \pi$ defines an isometry $\Psi$ of $X$ by sending $x\in B_v$ to the corresponding $x$ in $B_{\psi(v)}$. This simply shuffles the blocks around, acting by the identity map on the blocks. Note that the group of such automorphisms $\psi$ acts transitively on the vertices of $G$.

Combining these two types of isometries gives the desired result. In order to send a quadrilateral $Q\subset B_u$ to another quadrilateral $Q' \subset B_v$, first apply an isometry $\Psi$ as in the previous paragraph to send $B_u$ to $B_v$. Then move $\Psi(Q)$ to $Q'$ via an isometry $\Phi$ of the first type, preserving the block $B_v$.
\end{proof}

Since there are at most two isometries of $X$ sending one quadrilateral to another, this determines the isometry group of $X$ up to index $2$. We can reformulate this as follows. Subdivide $\calQ$ further into a tiling $\T$ by $(2,4,2q)$-triangles by bisecting the quadrilaterals at their smallest angle. Then the isometry group of $X$ may or may not act transitively on the tiles of $\T$ depending on the graphs $M$ and $G$ used to construct $X$.

 In \exref{example}, the isometry group of $X$ acts transitively on these triangles, but that is not the case in general. That is, there can be an asymmetry between the red and blue curves in $X$. For example, let $M$ be the flag-transitive map of type $\{4,4\}$ obtained by subdividing the square torus into a $5\times 5$ grid (so that $B$ is a torus with $25$ holes) and let $G$ be the $1$-skeleton of the $25$-dimensional cube. Then each component of $X \setminus \{\text{blue curves}\}$ is a torus with $16$ boundary components corresponding to a $4$-dimensional subcube of $G$, while the complementary components of the red curves are the blocks with $25$ boundary curves each. In this case, no isometry of $X$ can interchange the two families of systoles. 

\section{Critical point and index} \label{sec:critical}

A real-valued function $f$ on an $n$-dimensional manifold $M$ is a \emph{topological Morse function} if for every $p\in M$, there is an open neighborhood $U$ of $p$ and an injective continuous map $\phi : U \to \R^n$ with $\phi(p)=0$ such that $f\circ \phi^{-1}-f(p)$ takes either the form
\[(x_1,\ldots,x_n) \mapsto x_1 \]
or
\[(x_1,\ldots,x_n) \mapsto - \sum_{i=1}^j x_i^2 + \sum_{i=j+1}^n x_i^2 \]  
for some $j \in \{0,\ldots,n\}$. In the first case, $p$ is an \emph{ordinary point} and in the second case $p$ is \emph{a critical point of index $j$}. Critical points of index $0$ and $n$ are local minima and maxima respectively.

Let $g\geq 2$ and let $\T_g$ be the Teichm\"uller space of marked, connected, oriented, closed, hyperbolic surfaces of genus $g$. This space is a smooth manifold diffeomorphic to $\R^{6g-6}$. The systole $\sys(Y)$ of a surface $Y\in \T_g$ is the length of any of its shortest closed geodesics. As mentionned in the introduction, Akrout \cite{Akrout} proved that $\sys:\T_g \to \R_+$ is a topological Morse function. 

Let $Y\in \T_g$ and let $\calS$ be the set of (homotopy classes of) systoles in $Y$. For each $\alpha \in \calS$ and $Z \in \T_g$, we let $\ell_\alpha(Z)$ be the length of the unique closed geodesic homotopic to $\alpha$ in $Z$. These functions are differentiable on $\T_g$ and we denote their differentials by $d\ell_\alpha$. 

\begin{defn}
The point $Y\in \T_g$ is \emph{eutactic} if for every $v \in T_Y \T_g$, the following implication holds: if $d\ell_\alpha(v) \geq 0$ for every $\alpha \in \calS$, then $d\ell_\alpha(v)=0$ for every $\alpha \in \calS$. The \emph{rank} of a eutactic point $Y$ is the dimension of the image of the linear map ($d\ell_\alpha)_{\alpha\in\calS}: T_Y \T_g \to \R^\calS$.
\end{defn}

With these definitions, we have the following characterization of the cri\-ti\-cal points of $\sys$ \cite[Theorem 1]{Akrout}. 

\begin{thm}[Akrout]
The critical points of index $j$ of the systole function are the eutactic points of rank $j$.
\end{thm}

We can now show that the surface $X$ constructed in \secref{sec:gluing} is a critical point of $\sys$ and give an upper bound for its index.

\begin{prop}
Let $X$ be as in \secref{sec:gluing}. Then $X$ is a critical point of index at most $\frac{4q}{(q-2) p}(g-1)$ for the systole function.
\end{prop}
\begin{proof}
Let $\calS$ be the set of systoles of $X$ (the red curves and the blue curves). Suppose that $v\in T_X \T_g$ is such that $d\ell_\alpha(v) \geq 0$ for every $\alpha \in \calS$ and let 
\[
w = \sum_{f\in \isom(X)} f_*v.
\]
Then
\begin{equation} \label{eq:eutactic}
d\ell_\alpha(w) = \sum_{f\in \isom(X)} d\ell_\alpha(f_* v) = \sum_{f\in \isom(X)} d\ell_{f(\alpha)}(v) \geq d\ell_{\alpha}(v) \geq 0
\end{equation}
for every $\alpha \in \calS$. On the other hand, $w$ is the lift to $X$ of a deformation of the quotient orbifold $Q = X /\isom(X)$. By \propref{prop:isom}, $Q$ is either a $(2,4,2q)$-triangle or a $(2,2,2,q)$-quadrilateral. If $Q$ is a triangle, then $w=0$ so that $d\ell_\alpha(w)$ and $d\ell_{\alpha}(v)$ are both zero by \eqnref{eq:eutactic}, for every $\alpha \in \calS$. If $Q$ is a quadrilateral, then its deformation space is $1$-dimensional. This is because any $(2,2,2,q)$-quadrilateral is determined by the lengths $a$ and $b$ of the two sides disjoint from the angle $\pi/q$, which satisfy the relation
\[
\sinh a \sinh b = \cos(\pi/q)
\]
(see \cite[p.454]{Buser}). This equation implies that the lengths of the red curves and the blue curves in $\calS$ have opposite derivatives in the direction of $w$. Since the derivatives are non-negative, they must all be zero. We conclude that $d\ell_{\alpha}(v) = 0$ for every $\alpha \in \calS$ from \eqnref{eq:eutactic}. This shows that $X$ is eutactic. The number of systoles in $X$ is a trivial upper bound for the rank of $X$, and this number is equal to  $\frac{4q}{(q-2) p}(g-1)$ by \propref{prop:count}.
\end{proof}

Once again, by taking $p$ sufficiently large we obtain critical points of index at most $\eps g$ for any $\eps>0$, thereby proving \thmref{thm:critical}. This disproves the possibility envisaged by Schmutz Schaller \cite[p.410]{SchmutzMorse} that the minimal index were $2g-1$.

\section{Deformations preserving the systoles} \label{sec:dimension}
 
Let $\X_g$ be the subset of $\T_g$ whose systoles fill. We would like to show that $\X_g$ has relatively small codimension in $\T_g$. By \propref{prop:count}, the systoles of any surface $X$ constructed in \secref{sec:gluing} fill (recall that $X$ depends on several parameters). Let $\calS$ be the set of systoles in $X$. If we deform $X$ in such a way that the curves in $\calS$ remain of equal length, then these curves will still be the systoles for sufficiently small deformations. This is because the second shortest curve on $X$ is longer by a definite amount and length varies continuously.   

In other words, the intersection between the inverse image of the diagonal $\Delta \subset \R^\calS$ by the map $(\ell_\alpha)_{\alpha\in\calS}:\T_g\to\R^\calS$ and a small neighborhood of $X$ is contained in $\X_g$. One might be tempted to conclude directly that $\X_g$ has codimension at most $|\calS|-1$ in $\T_g$. The subtlety is that the image of $(\ell_\alpha)_{\alpha\in\calS}$ is not necessarily transverse to $\Delta$. Indeed, the rank of $X$ can be strictly less than $|\calS|-1$. For instance, the surface in \exref{example} has rank $2g-1$ according to \cite[Theorem 36]{SchmutzMorse}, while $|\calS|=2g+2$. 

To remedy this, one could try to get rid of redundant equations, i.e., to find a filling subset of curves $\calR \subset \calS$ for which the differential $(d_X \ell_\alpha)_{\alpha\in\calR}$ is surjective and apply the implicit function theorem. The problem is that even if the curves in $\calR$ stay of equal length, the curves in $\calS \setminus \calR$ might become shorter and so the systoles might not fill anymore.

Another approach would be to find a nearby surface $X_\theta$ which has the same set of systoles as $X$, and hope that the differential $(d_{X_\theta} \ell_\alpha)_{\alpha\in\calS}$ has full rank there. Below we will describe a $1$-dimensional family of deformations of $X$ with the same systoles. This fixes the issue of rank in some (but not all) cases. A similar idea was used in \cite{Sanki} to find a path of surfaces in $\X_g$ with $2g$ systoles.

\subsection*{A $1$-dimensional deformation}

Recall that $X$ is assembled from right-angled regular $2q$-gons $P$ whose sides are colored alternatingly red and blue, where $q \geq 3$. Given any $\theta \in (0,\pi)$, there exists a unique polygon $P_{\theta}$ with $2q$ equal sides and interior angles alternating between $\theta$ and $\pi - \theta$ (start with a triangle with angles $\pi/q$, $\theta /2$ and $(\pi - \theta)/2$ and reflect repeatedly across the two sides at angle $\pi/q$). To fix ideas, let us say that $\theta$ is the counter-clockwise (interior) angle from a red side to a blue side when going clockwise around $P_{\theta}$ and $\pi - \theta$ is the angle from blue to red. Now replace all the polygons $P$ in $X$ by $P_{\theta}$ while keeping the same gluing combinatorics. By construction, the total angle around vertices of the resulting tiling is $2\pi$ so the deformed surface $X_{\theta}$ is still a closed hyperbolic surface. Moreover, the red sides still line up to form closed geodesics and similarly for the blue sides. These closed geodesics all have equal length, namely, $p$ times the side length of $P_{\theta}$. As long as $\theta$ is close enough to $\pi/2$, these curves will remain the systoles. 

The goal is then to show that the linear map $(d_{X_\theta} \ell_\alpha)_{\alpha\in\calS}$ has full rank when $\theta \neq \pi/2$. We can do this for some small examples (see below), but we do not know how to handle surfaces with complicated gluing graphs of large girth. We present examples with full rank for girth $2$ and $3$ below.

\subsection*{Computing the rank}

To prove that the derivative of lengths has full rank, it suffices to find a set of tangent vectors $\{ v_\beta\}_{\beta \in \calS}$ to Teichm\"uller space for which the square matrix $(d_{X_\theta} \ell_\alpha(v_\beta))_{\alpha,\beta\in\calS}$ has non-zero determinant. For this, we can choose each vector $v_\beta$ to be the Fenchel-Nielsen twist deformation (i.e., left earthquake) around the curve $\beta$. The cosine formula of Wolpert \cite{WolpertTwist} and Kerckhoff \cite{Kerckhoff} then says that
\[
d\ell_\alpha(v_\beta) = \sum_{p \in \alpha \cap \beta} \cos \angle_ p(\alpha,\beta)
\]
whenever $\alpha$ and $\beta$ are transverse, where $\angle_ p(\alpha,\beta)$ is the counter-clockwise angle from $\alpha$ to $\beta$ at the point $p$. In our case, two distinct curves $\alpha, \beta \in \calS$ intersect at most once, with angle $\theta$ from red to blue or $\pi - \theta$ from blue to red. If we split the rows and columns of $D=(d_{X_\theta} \ell_\alpha(v_\beta))_{\alpha,\beta\in\calS}$ by color we get a block matrix of the form
\[
D = \cos \theta \begin{pmatrix}
0 & A \\
-A^\intercal & 0
\end{pmatrix}
\]
where $A$ is the matrix of zeros and ones recording which red curves intersect which blue curves. If $\theta \neq \pi / 2$, then $D$ has full rank if and only if the matrix
\[
\wtilde D = \begin{pmatrix}
0 & A \\
A^\intercal & 0
\end{pmatrix}
\]
does. This matrix is the adjacency matrix of some graph $I_\calS$, namely, the graph whose vertices are the systoles of $X$ and where two vertices are joined by an edge if and only if the corresponding systoles intersect.

The determinant of the adjacency matrix of a graph counts something combinatorial on the graph. Indeed, according to \cite{Harary} we have 
\[
\det(\wtilde D) = \sum_{J \subset I_\calS} (-1)^{\#\{\text{even components of }J\}} 2^{\#\{\text{cycles in }J\}}
\]
where the sum is over all spanning subgraphs of $J\subset I_\calS$ (subgraphs containing all vertices) which are \emph{elementary}, meaning that their components are either edges or embedded cycles. The even components are those with an even number of vertices. In our case, $I_\calS$ is bipartite so that all its cycles are even. 

We are now ready to give some examples where $\wtilde D$ is invertible. 

\subsection*{Examples of girth 2}

The first family of examples comes from \exref{example}. In that example, the red and blue curves form a chain, that is, $I_\calS$ is a cycle of length $2g+2$. It follows that $I_\calS$ has exactly three elementary spanning subgraphs: $I_\calS$ itself, and two subgraphs obtained by deleting every other edge in $I_\calS$. If $g = 2m$ is even, then $I_\calS$ has $4m+2$ edges and
\begin{equation} \label{eq:det}
\det(\wtilde D) = (-1)^1 2^1 + (-1)^{2m+1} 2^0 + (-1)^{2m+1} 2^0 = -4 \neq 0.
\end{equation}
Alternatively, one could compute the determinant of $\wtilde D$ by using the fact that $A$ is a circulant matrix in this case.

The fact that $\wtilde D$ has non-zero determinant implies that the derivative of $\ell=(\ell_\alpha)_{\alpha\in\calS}:\T_g\to\R^\calS$ has full rank at $X_\theta$ whenever $\theta \neq \pi /2$. By the implicit function theorem, near $X_\theta$ we have that $\ell^{-1}(\Delta)$ is a smooth submanifold of codimension \[|\calS|-1=(2g+2)-1 = 2g+1,\] hence of dimension $4g-7$. As explained earlier, $\ell^{-1}(\Delta)$ intersected with a sufficiently small ball around $X_\theta$ is contained in $\X_g$. We have thus proved that $\X_g$ has dimension at least $4g-7$ when $g$ is even. We can push the proof a little further to obtain the following.

\begin{thm} \label{thm:dimension}
For every even $g \geq 2$, the set $\X_g \subset \T_g$ of closed hyperbolic surfaces of genus $g$ whose systoles fill contains a cell of dimension $4g-5$. 
\end{thm}
\begin{proof}
Let $X$ be the surface of genus $g$ from \exref{example} and let \[\calS = \{\alpha_1, \ldots, \alpha_{2g+2}\}\] be its set of systoles labelled in such a way that $\alpha_j$ intersects $\alpha_{j-1}$ and $\alpha_{j+1}$ for every $j$, where the indices are taken modulo $2g+2$. Let $X_\theta$ be the deformation of $X$ described above, where $\theta$ is close enough to $\pi/2$ so that its sets of systoles is still equal to $\calS$. By Equation \eqref{eq:det}, the map $\ell=(\ell_\alpha)_{\alpha\in\calS}:\T_g\to\R^\calS$ is a submersion at the point $X_\theta$. In particular, $\ell$ is open in a neighborhood of $X_\theta$. This implies that there exist surfaces $Y$ arbitrarily close to $X_\theta$ such that
\[
\ell_{\alpha_1}(Y) = \cdots = \ell_{\alpha_{2g}}(Y)
\]
and such that these lengths are strictly less than $\ell_{\alpha_{2g+1}}(Y)$ and $\ell_{\alpha_{2g+2}}(Y)$. If $Y$ is close enough to $X_\theta$, then its set of systoles is a subset of $\calS$ by continuity of the length functions. Therefore, there is a sequence $Y_n$ converging to $X_\theta$ such that the systoles in $Y_n$ are given by the set $\calR = \{ \alpha_1 , \ldots, \alpha_{2g}\}$. 

If $n$ is large enough, then the square matrix $(d_{Y_n}\ell_\alpha(v_\beta))_{\alpha,\beta \in \calR}$ will have non-zero determinant. Indeed, in the limit the matrix has the form 
\[
(d_{X_\theta}\ell_\alpha(v_\beta))_{\alpha,\beta \in \calR} = \cos \theta \begin{pmatrix}
0 & B \\
-B^\intercal & 0
\end{pmatrix}
\]
which is invertible because $\begin{pmatrix}
0 & B \\
B^\intercal & 0
\end{pmatrix}$ is the adjacency matrix of a tree $I_\calR$ with an even number of vertices. Up to sign, its determinant is the number of perfect matchings (spanning subgraphs whose components are edges) in $I_\calR$, which is equal to one. Furthermore, the entries of $(d_{Y}\ell_\alpha(v_\beta))_{\alpha,\beta \in \calR}$ depend continuously on the surface $Y$ in the same way that the angles of intersection between geodesics do.

Let $Y=Y_n$ for any such large enough $n$. Then the systoles in $Y$ are given by the set $\calR$ and the map $(\ell_\alpha)_{\alpha\in\calR}:\T_g\to\R^\calR$ is a submersion at $Y$. By the implicit function theorem, the inverse image of the diagonal by this map is a submanifold of codimension $2g-1$ near $Y$. Since the curves in $\calR$ fill, a small neighborhood of $Y$ in this submanifold is contained in $\X_g$. The curves in $\calR$ fill because the complement of the curves in $\calS$ is a union of four polygons which meet at the intersection of $\alpha_{2g+1}$ and $\alpha_{2g+2}$. Adding these two curves fuses the four polygons into a single one.   
\end{proof}

When $g$ is odd, the matrix $\wtilde D$ is singular, but this does not ne\-cessarily imply that the image of $(\ell_\alpha)_{\alpha \in \calS}$ is not transverse to the diagonal.

\subsection*{An example of girth 3}

Next, we present an example of genus $g=6$ where the underlying graphs $M$ and $G$ for the surface $X$ have girth $3$ and the matrix $\wtilde D$ is non-singular. Let $M$ be the $1$-skeleton of a regular tetrahedron (as a map of type $\{3,3\}$ on the sphere) and let $B$ be the corresponding block. This is a sphere with $4$ holes and tetrahedral symmetry. Although this does not fit in the theory of \secref{sec:gluing}, it is possible to glue five copies of $B$ along the complete graph $K_5$ in such a way that the blue arcs connect up in groups of three to form closed geodesics. To see this, it is convenient to draw $K_5$ in $\R^3$ with a $3$-fold symmetry as in \figref{figure}.
\begin{figure}[htp]
\centering
{\includegraphics[width=\textwidth]{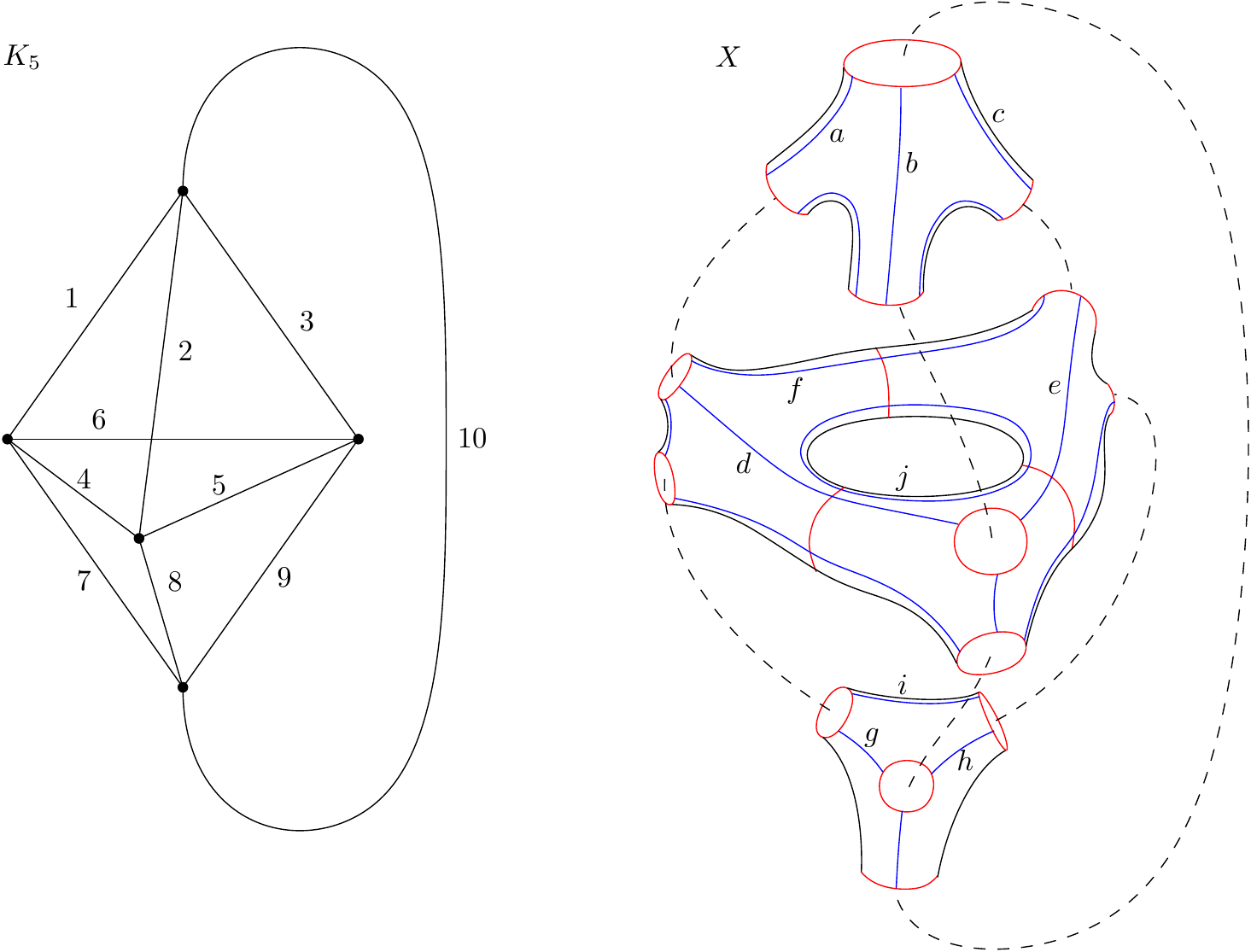}}
\caption{An embbeding of $K_5$ in $\R^3$ and the corresponding gluing of tetrahedral blocks.}\label{figure}
\end{figure}

The tetrahedral pieces are glued as suggested by the figure, in the simplest possible way (without twist). By inspection, the blue arcs connect in groups of three. The proof of \propref{prop:count} applies without change to show that the systoles in $X$ are the red curves and the blue curves. The genus of $X$ is equal to the number of edges in the complement of any spanning tree in $K_5$, which is $10-4=6$. Let us label the red curves from $1$ to $10$ and the blue curves from $a$ to $j$ as in \figref{figure} (the red curves correspond to the edges in $K_5$). Then the intersection matrix $A$ is given by
\[
A = \begin{pmatrix}
1 & 0 & 0 & 1 & 0 & 1 & 0 & 0 & 0 & 0\\ 
0 & 1 & 0 & 1 & 1 & 0 & 0 & 0 & 0 & 0\\ 
0 & 0 & 1 & 0 & 1 & 1 & 0 & 0 & 0 & 0\\ 
0 & 0 & 0 & 1 & 0 & 0 & 1 & 0 & 0 & 1\\ 
0 & 0 & 0 & 0 & 1 & 0 & 0 & 1 & 0 & 1\\ 
0 & 0 & 0 & 0 & 0 & 1 & 0 & 0 & 1 & 1\\ 
1 & 0 & 0 & 0 & 0 & 0 & 1 & 0 & 1 & 0\\ 
0 & 1 & 0 & 0 & 0 & 0 & 1 & 1 & 0 & 0\\ 
0 & 0 & 1 & 0 & 0 & 0 & 0 & 1 & 1 & 0\\ 
1 & 1 & 1 & 0 & 0 & 0 & 0 & 0 & 0 & 0\\ 
\end{pmatrix}
\]
which has determinant $48 \neq 0$. Therefore, the derivative of lengths $d\ell$ has full rank at $X_\theta$ whenever $\theta \neq \pi /2$. Note that this only gives us that $\X_g$ has codimension at most $19$ in $\T_g$, hence dimension at least $11 = 4g - 13$. The conclusion is weaker than that of \thmref{thm:dimension}, but we wanted to include this example to show that $\widetilde D$ can have full rank for more complicated graphs.

\subsection*{Questions}

We conclude with a few questions related to the strategy we have just outlined.

\begin{question}
Is there a sequence of graphs $M$ and $G$ as in \secref{sec:gluing} with girth going to infinity such that the corresponding intersection matrices $\wtilde D$ have non-zero determinants?
\end{question}

In view of the above reasoning and the counting of \propref{prop:count}, a po\-si\-tive answer would imply \conjref{conj:dimension}. A major difficulty is that $M$ and $G$ are given to us in a non-explicit way from \thmref{thm:reg_map} and \thmref{thm:poly_graph}.

As the proof of \thmref{thm:dimension} shows, one could bypass the determinant issue by fin\-ding a filling subset $\calR \subset \calS$ of even cardinality such that the corresponding intersection graph $I_\calR$ is a tree, and a surface $Y$ near $X_\theta$ whose systoles are exactly the curves in $\calR$.

\begin{question}
Given a surface $X$ constructed as in \secref{sec:gluing} with set of systoles $\calS$, is there an induced subtree in $I_\calS$ with an even number of vertices such that the union of the corresponding curves fill? 
\end{question}

\begin{question}
Let $X$ be any hyperbolic surface, let $\calS$ be its set of systoles and let $\calR \subset \calS$ be a non-empty subset. Does there exist, in every neighborhood of $X$, a surface whose set of systoles is equal to $\calR$?
\end{question}

Even if these questions have negative answers, they suggest how one should modify the construction of surfaces with sublinearly many systoles that fill in order to show that $\X_g$ has large dimension: the systoles should cut the surface into a single polygon instead of several.

\bibliographystyle{amsalpha}
\bibliography{biblio}

\end{document}